\documentclass[12pt]{amsart}
\usepackage{amssymb}

\theoremstyle{plain}
\newtheorem{theorem}{Theorem}[section]
\newtheorem{corollary}[theorem]{Corollary}
\newtheorem{lemma}[theorem]{Lemma}
\newtheorem{proposition}[theorem]{Proposition}

\theoremstyle{definition}
\newtheorem{definition}[theorem]{Definition}
\newtheorem{remark}[theorem]{Remark}

\theoremstyle{remark}

\newcommand{\abs}[1]{\lvert#1\rvert}
\newcommand{\norm}[1]{\lVert#1\rVert}
\newcommand{\bigabs}[1]{\bigl\lvert#1\bigr\rvert}

\newcommand{\Bigabs}[1]{\Bigl\lvert#1\Bigr\rvert}

\renewcommand{\le}{\leqslant}
\renewcommand{\ge}{\geqslant}
\renewcommand{\mid}{\::\:}

\newcommand{\term}[1]{{\textit{\textbf{#1}}}}

\DeclareMathOperator{\Span}{span}
\DeclareMathOperator{\codim}{codim}
\DeclareMathOperator{\dist}{dist}

\begin{document}
\baselineskip 18pt

\title[Almost invariant half-spaces]
       {Almost invariant half-spaces of operators on Banach spaces}

\author[G.~Androulakis]{George Androulakis}
\author[A.I.~Popov]{Alexey I. Popov}
\author[A.~Tcaciuc]{Adi Tcaciuc}
\author[V.G.~Troitsky]{Vladimir G. Troitsky}

\address[G.Androulakis]{Department of Mathematics,
          University of South Carolina, Columbia, SC 29208, USA}
\email{giorgis@math.sc.edu}

\address[A.~I. Popov, and V.~G. Troitsky]{Department of Mathematical
  and Statistical Sciences, University of Alberta, Edmonton,
  AB, T6G\,2G1. Canada}
\email{apopov@math.ualberta.ca, vtroitsky@math.ualberta.ca}

\address[A. Tcaciuc]{Mathematics and Statistics Department,
   Grant MacEwan College, Edmonton, Alberta, Canada T5J
   P2P, Canada}
\email{tcaciuc@math.ualberta.ca}

\thanks{The third and the fourth authors were supported by NSERC}

\date{\today.}

\begin{abstract}
  We introduce and study the following modified version of the
  Invariant Subspace Problem: whether every operator $T$ on a Banach
  space has an almost invariant half-space, that is, a subspace $Y$ of
  infinite dimension and infinite codimension such that $Y$ is of
  finite codimension in $T(Y)$. We solve this problem in the
  affirmative for a large class of operators which includes
  quasinilpotent weighted shift operators on $\ell_p$ ($1\le
  p<\infty$) or $c_0$.
\end{abstract}

\maketitle

\section{Introduction}\label{intro}

Throughout the paper, $X$ is a Banach space and by ${\mathcal L}(X)$ we
denote the set of all (bounded linear) operators on~$X$. By a
``subspace'' of a Banach space we always mean a ``closed subspace''.
Given a sequence $(x_n)$ in~$X$, we write $[x_n]$ for the closed
linear span of $(x_n)$.

\begin{definition}
  A subspace $Y$ of a Banach space $X$ is called a \term{half-space}
  if it is both of infinite dimension and of infinite codimension in~$X$.
\end{definition}

\begin{definition}\label{AI}
  If $T\in\mathcal{L}(X)$ and $Y$ is a subspace of~$X$, then $Y$ is
  called \term{almost invariant} under~$T$, or
  \term{$T$-almost invariant}, if there exists a finite dimensional
  subspace $F$ of $X$ such that $T(Y)\subseteq Y+F$.
\end{definition}

In this work, the following question will be referred to as the
\term{almost invariant half-space problem}: \emph{Does every operator on a
Banach space have an almost invariant half-space?} Observe that every
subspace of $X$ that is not a half-space is clearly almost invariant
under any operator. Also, note that the almost invariant half-space
problem is not weaker than the well known invariant subspace problem,
because in the latter the invariant subspaces are not required to be
half-spaces.

The natural question whether the usual unilateral right shift operator
acting on a Hilbert space has almost invariant half-spaces has an
affirmative answer. Moreover, it is known that this operator has even
invariant half-spaces. Indeed, by \cite[Corollary~3.15]{Radjavi:03},
this operator has an invariant subspace with infinite-dimensional
orthogonal complement (thus the invariant subspace is of infinite
codimension). It is not hard to see that the space exhibited in the
proof of this statement is in fact infinte dimensional.

It is natural to consider Donoghue operators as candidates for
counter\-examples to the almost invariant half-space problem, as their
invariant subspaces are few and well understood. Recall that a
Donoghue operator $D \in {\mathcal L}(\ell_2)$ is an operator defined
by
\begin{displaymath}
  De_0=0,\quad De_i=w_ie_{i-1},\quad i\in\mathbb N,
\end{displaymath}
where $(w_i)$ is a sequence of non-zero complex numbers such that
$\bigl(\abs{w_i}\bigr)$ is monotone decreasing and in~$\ell_2$.  It is
known that if $D$ is a Donoghue operator then $D$ has only invariant
subspaces of finite dimension and $D^*$ has only invariant subspaces
of finite codimension (see \cite[Theorem~4.12]{Radjavi:03}). Hence
neither $D$ nor $D^*$ have invariant half-spaces. In
Section~\ref{w-shifts} we will employ the tools of Section~\ref{tools}
to show that, nevertheless, every Donoghue operator has almost
invariant half-spaces. We do not know whether the operators
constructed by Enflo~\cite{Enflo:87} and Read~\cite{Read:84} have
almost invariant half-spaces.

The following result explains how almost invariant half-spaces of operators are 
related to invariant subspaces of perturbed operators.

\begin{proposition}
  Let $T\in\mathcal{L}(X)$ and $H\subseteq X$ be a half-space. Then
  $H$ is almost invariant under $T$ if and only if $H$ is invariant
  under $T+K$ for some finite rank operator~$K$.
\end{proposition}

\begin{proof}
  Suppose that $T$ has an almost invariant half-space~$H$. Let $F$ be
  a subspace of the smallest dimension satisfying the condition in
  Definition~\ref{AI}. Then we have $H\cap F=\{0\}$. Define $P\colon
  H+F\to F$ by $P(h+f)=f$.  Since $P$ is a finite rank operator, we
  can extend it to a finite rank operator on $X$ using Hahn-Banach
  theorem.  That is, there exists $\widetilde{P}\colon X\to F$ such that
  $\widetilde{P}|_{H+F}=P$. Define $K\colon X\to X$ by $K:=-\widetilde{P}T$.
  Clearly $K$ has finite rank and for any $h\in H$ we have $Th=h'+f$
  for some $h'\in H$ and $f\in F$, so that
  \begin{displaymath}
   (T+K)(h)=Th-\widetilde{P}Th=h'+f-\widetilde{P}(h'+f)=h'+f-f=h'
  \end{displaymath}
  Therefore, $(T+K)H\subseteq H$, which shows that $T+K$ has an
  invariant half-space.

  Conversely, from $(T+K)(H)\subseteq H$ it follows immediately that
  $T(H)\subseteq H+K(H)$,so that $H$ is an almost invariant half space
  for~$T$.
\end{proof}

Finally we would like to point out that if an operator has almost
invariant half-spaces, then so does its adjoint. For that we will need
two simple lemmas. The proof of the first lemma is elementary.

\begin{lemma}\label{hs-criterion}
  Let $X$ be a Banach space and $Y$ be a subspace of~$X$. Then $Y$ is
  infinite codimensional if and only if $Y^\perp$ is of infinite
  dimension. Thus $Y$ is a half-space if and only if both $Y$ and
  $Y^\perp$ are of infinite dimension.
\end{lemma}

\begin{lemma}\label{hs-annihilator}
  A subspace $Y$ of $X$ is a half-space if and only if $Y^\perp$ is a
  half-space in~$X^*$.
\end{lemma}

\begin{proof}
  Suppose $Y$ is a half-space. By Lemma~\ref{hs-criterion}, $Y^\perp$
  must be infinite-dimensional. Also $(Y^\perp)^\perp \supseteq j(Y)$
  where $j\colon X \to X^{**}$ denotes the natural embedding. Thus
  $(Y^\perp)^\perp$ is infinite dimensional. Now
  Lemma~\ref{hs-criterion} yields that $Y^\perp$ is a half-space.

  Let's assume that $Y^\perp$ is a half-space. Since $Y^\perp$ is
  infinite codimensional $Y$ must be infinite-dimensional (see,
  e.g. \cite[Theorem~5.110]{Aliprantis:06}).  On the other hand, since
  $Y^\perp$ is infinite dimensional, by Lemma~\ref{hs-criterion} we
  obtain that $Y$ is of infinite codimension, thus a half-space.
\end{proof}

\begin{remark}
  The statement dual to that of Lemma~\ref{hs-annihilator} is not
  true in general. That is, if $Z$ is a half-space in $X^*$ then $Z_\perp$ need
  not be a half-space. For example, $c_0$ is a half-space in
  $\ell_\infty$ while $(c_0)_\perp=\{0\}\subseteq\ell_1$ is not.
\end{remark}

\begin{proposition}\label{aihs-adjoints}
  Let $T$ be an operator on a Banach space~$X$. If $T$ has an almost
  invariant half-space then so does its adjoint~$T^*$.
\end{proposition}

\begin{proof}
  Let $Y$ be a half-space in $X$ such that $Y$ is almost invariant
  under~$T$, and $F$ be a finite-dimensional subspace of $X$ of
  smallest dimension such that $TY\subseteq Y+F$. Then $Y \cap F = \{
  0 \}$. Thus there exists a subspace $W$ of $X$ such that $W+F=X$,
  $W\cap F=\{0\}$, and $Y \subseteq W$.  In particular, $W^\perp$ is
  finite dimensional.  Denote $Z=(Y+F)^\perp$. By
  Lemma~\ref{hs-annihilator}, $Z$ is a half-space in~$X^*$. For every
  $z\in Z$ and $y\in Y$ we have
  \begin{math}
    \langle y,T^*z\rangle=\langle Ty,z\rangle=0
  \end{math}
  since $Ty\in Y+F$. Therefore 
  \begin{math}%\label{adjoints1}
    T^*Z\subseteq Y^\perp.
  \end{math}
  To finish the proof, it suffices to show that
  \begin{math}%\label{adjoints2}
    Y^\perp= Z+W^\perp.
  \end{math}

  Indeed, by the definition of $Z$ we have that $Z \subseteq Y^\perp$.
  Also since $Y \subseteq W$ we have $W^\perp \subseteq Y^\perp$. Thus
  $Z+W^\perp \subseteq Y^\perp$. On the other hand, since $F$ is
  finite dimensional and $F \cap W = \{ 0 \}$, we may choose a basis
  $(f_i)$ of $F$ with biorthogonal functionals $(f_i^*)$ such that
  $f_i^* \in W^\perp$. Since $Y \subseteq W$ we have that $f_i^* \in
  Y^\perp$. Thus, if $x^*$ is an arbitrary element of $Y^\perp$ then
  $x^* - \sum_i\limits x^*(f_i)f_i^* \in (Y+F)^\perp=Z$, and therefore
  $x^* \in Z+W^\perp$.
%  which finishes the proof of (\ref{adjoints2}).
%  Obviously (\ref{adjoints1}) and (\ref{adjoints2}) finish the proof
%  of the proposition.
\end{proof}

\section{Basic tools}\label{tools}

All Banach spaces in Sections~\ref{tools}, \ref{w-shifts} and
\ref{ext} are assumed to be complex. For a subset $A$ of~$\mathbb C$,
we will write $A^{-1}=\big\{\frac{1}{\lambda}\colon\lambda\in
A,\lambda\ne 0\big\}$.  For a Banach space $X$ and $T \in
{\mathcal L}(X)$, we will use symbols $\sigma(T)$ for the spectrum
of~$T$, $r(T)$ for the spectral radius of~$T$, and $\rho(T)$ for the
resolvent set of~$T$.  For a nonzero vector $e\in X$ and
$\lambda\in \rho(T)^{-1}$, define a vector $h(\lambda,e)$ in $X$ by
\begin{displaymath}
   h(\lambda, e):=\bigl(\lambda^{-1}I-T\bigr)^{-1}(e).
\end{displaymath}
Note that if $\abs{\lambda} < \frac{1}{r(T)}$ then% 
\footnote{In case $r(T)=0$ we take $\frac{1}{r(T)}=+\infty$.}
Neumann's formula yields
\begin{equation}
  \label{Neumann}
  h(\lambda,e) = \lambda \sum_{n=0}^\infty \lambda^n T^n.
\end{equation}
Also, observe that $\bigl(\lambda^{-1}I-T\bigr)h(\lambda, e)=e$ for
every $\lambda\in \rho(T)^{-1}$, so that
\begin{equation}
  \label{Th}
  Th(\lambda,e)=\lambda^{-1}h(\lambda,e)-e.
\end{equation}
The last identity immediately yeilds the following result.

\begin{lemma}\label{h_lambda-ai}
  Let $X$ be a Banach space, $T \in {\mathcal L}(X)$, $0\ne e\in X$,
  and $A\subseteq \rho(T)^{-1}$. Put
  \begin{displaymath}
    Y=\overline{\Span}\bigl\{h(\lambda,e)\colon\lambda\in A\bigr\}.
  \end{displaymath}
  Then $Y$ is a $T$-almost invariant subspace (which is not not
  necessarily a half-space), with $TY\subseteq Y + \Span\{e\}$.
\end{lemma}

\begin{remark}
  \textit{The Replacement procedure.} For any nonzero vector $e$ in a Banach
  space~$X$, we have
  \begin{displaymath}
    h(\lambda,e)-h(\mu,e)=(\mu^{-1}-\lambda^{-1})h\big(\lambda,h(\mu,e)\big)
  \end{displaymath}
  whenever $\lambda,\mu \in\rho(T)^{-1}$.

  Indeed, 
  \begin{eqnarray*}
    h(\lambda,e)-h(\mu,e)
    &=&\bigl[(\lambda^{-1}I-T)^{-1}-(\mu^{-1}I-T)^{-1}\bigr](e)\\
    &=&(\mu^{-1}-\lambda^{-1})(\lambda^{-1}I-T)^{-1}(\mu^{-1}I-T)^{-1}(e)\\
    &=&(\mu^{-1}-\lambda^{-1})(\lambda^{-1}I-T)^{-1}h(\mu,e)\\
    &=&(\mu^{-1}-\lambda^{-1})h\big(\lambda,h(\mu,e)\big)
  \end{eqnarray*}
\end{remark}

\begin{lemma}\label{lin-indep}
  Suppose that $T\in\mathcal{L}(X)$ has no eigenvectors. Then,
  for any nonzero vector $e\in X$ the set $\bigl\{h(\lambda,e) \mid
  \lambda \in\rho(T)^{-1}\bigr\}$ is linearly independent.
\end{lemma}

\begin{proof}
  We are going to use induction on $n$ to show that for any nonzero
  vector $e\in X$ and any distinct $\lambda_1, \lambda_2, \dots,
  \lambda_n\in\rho(T)^{-1}$ the set
  $$\bigl\{h(\lambda_1,e), h(\lambda_2,e),\dots, h(\lambda_n,e)\bigr\}$$
  is linearly independent. The statement is clearly true for $n=1$; we
  assume it is true for $n-1$ and will prove it for~$n$.

  Fix $e\in X$ and distinct $\lambda_1, \lambda_2, \dots,
  \lambda_n\in\rho(T)^{-1}$.  Let $a_1, a_2, \dots, a_n$ be scalars
  such that $\sum_{k=1}^{n}a_k h(\lambda_k,e)=0$. It follows
  from~\eqref{Th} that
  \[
    0=T\biggl(\sum_{k=1}^{n}a_k h(\lambda_k,e)\biggr)
     =\sum_{k=1}^n a_k\lambda_k^{-1}h(\lambda_k,e)-\sum_{k=1}^{n}a_k e.
  \]
  If $\sum_{k=1}^{n}a_k\ne 0$ then 
  $e\in\Span\big\{h(\lambda_k,e)\big\}_{k=1}^{n}$, so that
  $\Span\bigl\{h(\lambda_k,e)\bigr\}_{k=1}^{n}$ is $T$-invariant by~\eqref{Th}.
  This subspace is finite-dimensional, so that
  $T$ has an eigenvalue, which is a contradiction.
  Therefore $\sum_{k=1}^{n}a_k=0$, so that $a_1=-\sum_{k=2}^na_k$.

  Using the Replacement Procedure we obtain
  \begin{eqnarray*}
     0=\sum_{k=1}^{n}a_k h(\lambda_k,e)&=&
     \Big(-\sum_{k=2}^{n}a_k\Big)h(\lambda_1,e)
        +\sum_{k=2}^{n}a_k h(\lambda_k,e)\\
     &=&\sum_{k=2}^{n}a_k\big(h(\lambda_k,e)-h(\lambda_1,e)\big)\\
     &=&\sum_{k=2}^{n}a_k(\lambda_1^{-1}-\lambda_k^{-1})h
            \bigl(\lambda_k,h(\lambda_1,e)\bigr).
  \end{eqnarray*}
  By the induction hypothesis, the set
  $\Bigl\{h\bigl(\lambda_k,h(\lambda_1,e)\bigr)\Bigr\}_{k=2}^{n}$ is
  linearly independent, hence $a_k(\lambda_1^{-1}-\lambda_k^{-1})=0$
  for any $2\le k\le n$. It follows immediately that $a_k=0$ for any
  $1\le k\le n$, and this concludes the proof.
\end{proof}

This gives us a natural way to try to construct almost invariant
half-spaces. Indeed, suppose that $T$ has no eigenvectors. Let $e\in
X$ such that $e\ne 0$, and let $(\lambda_n)$ be a sequence of distinct
elements of $\rho(T)^{-1}$. Put
$Y=\bigl[h(\lambda_n,e)\bigr]_{n=1}^\infty$. Then $Y$ is almost
invariant by Lemma~\ref{h_lambda-ai} and infinite-dimensional by
Lemma~\ref{lin-indep}. However, the difficult part is to show that $Y$
is infinite codimensional.  Even passing to subsequences might not
help, as there are sequences whose every subsequence spans a dense
subspace. For example, let $x=(1,1,\frac{1}{2!},\frac{1}{3!},\dots)$
in $X=\ell_p$ with $1\le p<\infty$, and put $x_n=S^nx$ where $S$ is
the backward shift operator. Let $Y$ be the closed subspace spanned by
a subsequence of $(x_n)$. We claim that $Y=X$. Indeed, $n!x_n\to e_1$,
so that $e_1\in Y$.  Let $y_n=n!x_n-e_1$; it is easy to see that
$(n+1)!y_n\to e_2$, so that $e_2\in Y$.  Proceding inductively, we
obtain that $e_i\in Y$ for every~$i$, hence $Y=X$.

\section{Weighted shift operators}\label{w-shifts}

In this section we give a sufficient condition for a quasinilpotent
operator to have almost invariant half-spaces (Theorem~\ref{main}). As
an application, we show in Corollary~\ref{aihs-shifts} that
quasinilpotent weighted shifts on $\ell_p$ or $c_0$ have invariant
half-spaces. In particular, every Donoghue operator has an almost
invariant half-space.

Recall that a sequence $(x_i)$ in a Banach space is called
\term{minimal} if $x_k\notin[x_i]_{i\ne k}$ for every~$k$, (see also
\cite[section 1.f]{Lindenstrauss:77}).  It is easy to see that this is
equivalent to saying that for every $k$, the biorthogonal functional
$x_k^*$ defined on $\Span\{x_i\}$ by
$x_k^*\bigl(\sum_{i=0}^n\alpha_ix_i\bigr)=\alpha_k$ is bounded.

We will use the following numerical lemma.

\begin{lemma}\label{converge}
  Given a sequence $(r_i)$ of positive reals, there exists a sequence
  $(c_i)$ of positive reals such that the series $\sum_{i=0}^\infty
  c_ir_{i+k}$ converges for every~$k$.
\end{lemma}

\begin{proof}
  For every $i$ take
  $c_i=\frac{1}{2^i}\min\{\frac{1}{r_1},\dots,\frac{1}{r_{2i}}\}$. For
  every $i\ge k$ we have $k+i\le 2i$, so that
  $c_ir_{i+k}\le\frac{1}{2^i}$. It follows that
  \begin{displaymath}
    \sum_{i=0}^\infty c_ir_{i+k}\le
    \sum_{i=0}^{k-1} c_ir_{i+k}+
    \sum_{i=k}^\infty\tfrac{1}{2^i}<+\infty.
  \end{displaymath}
\end{proof}

\begin{theorem}\label{main}
  Let $X$ be a Banach space and $T \in {\mathcal L}(X)$ satisfying the
  following conditions:
  \begin{enumerate}
  \item $T$ has no eigenvalues.
  \item\label{rho} The unbounded component of $\rho(T)$ contains
      $\{z\in\mathbb C\mid 0<\abs{z}<\varepsilon\}$ for some $\varepsilon>0$.
  \item There is a vector whose orbit is a minimal sequence.
  \end{enumerate}
  Then $T$ has an almost invariant half-space.
\end{theorem}

\begin{proof}
  Let $e\in X$ be such that $(T^ie)_{i=0}^\infty$ is minimal. For each
  $i$ put $x_i=T^ie$.  Then for each~$k$, the biorthogonal functional
  $x_k^*$ defined on $\Span x_i$ by
  $x_k^*\bigl(\sum_{i=0}^n\alpha_ix_i\bigr)=\alpha_k$ is bounded. Let
  $r_i=\norm{x_k^*}$. Let $(c_i)$ be a sequence of positive real
  numbers as in Lemma~\ref{converge}, so that
  $\beta_k:=\sum_{i=0}^\infty c_ir_{i+k}<+\infty$ for every~$k$. By
  making $c_i$'s even smaller, if necessary, we may assume that
  $\sqrt[i]{c_i}\to 0$.
  
  Consider a function $F:\mathbb C\to\mathbb C$ defined by
  $F(z)=\sum_{i=0}^\infty c_iz^i$. Evidently, $F$ is entire.
  Observe that we may assume that the set 
  $\bigl\{z\in\mathbb C\mid F(z)=0\bigr\}$ is infinite.
  Indeed, by Picard Theorem there exists a negative real number $d$
  such that the set $\bigl\{z\in\mathbb C\mid F(z)=d\bigr\}$ is
  infinite. Now replace $c_0$ with $c_0-d$. This doesn't affect our other
  assumptions on the sequence $(c_i)$.

  Fix a sequence of distinct complex numbers $(\lambda_n)$ such that
  $F(\lambda_n)=0$ for every~$n$. Since $F$ is non-constant, the
  sequence $(\lambda_n)$ has no accumulation points. Hence,
  $\abs{\lambda_n}\to +\infty$. 
  
  Note that~\eqref{rho} can be restated as follows: $\rho(T)^{-1}$ has
  a connected component ${\mathcal C}$ such that $0 \in
  \overline{\mathcal C}$ and ${\mathcal C}$ contains a neighbourhood
  of~$\infty$.  Thus by passing to a subsequence of $\lambda_n$'s and
  relabeling, if necessary, we can assume that $\lambda_n \in
  {\mathcal C}$ for all~$n$.

  Observe that the condition $\lambda_n\in\rho(T)^{-1}$ for every $n$
  implies that $h(\lambda_n,e)$ is defined for each~$n$. Put
  $Y=[h(\lambda_n,e)]_{n=1}^\infty$.
  Then $Y$ is almost invariant under $T$ by Lemma~\ref{h_lambda-ai}
  and $\dim Y=\infty$ by Lemma~\ref{lin-indep}.
  We will prove that $Y$ is actually a half-space
  by constructing a sequence of linearly independent functionals $(f_n)$
  such that every $f_n$ annihilates~$Y$.

  For every $k=0,1,\dots$, put $F_k(z)=z^kF(z)$. Let's write $F_k(z)$ in 
  the form of Taylor series, $F_k(z)=\sum_{i=0}^\infty c^{(k)}_iz^i$. Then
  \begin{displaymath}
    c^{(k)}_i=
    \begin{cases}
      0 & \mbox{if $i<k$, and}\\
      c_{i-k} & \mbox{if }i\ge k.
    \end{cases}
  \end{displaymath}

  Define a functional $f_k$ on $\Span\{T^ie\}_{i=0}^\infty$ via
  $f_k(T^ie)=c^{(k)}_i$. Since $T$ has no eigenvalues, the orbit of $T$ is linearly
  independent thus $f_k$ is well-defined.
  We will show now that $f_k$ is bounded. Let
  $x\in\Span\{T^ie\}_{i=0}^\infty$, then $x=\sum_{i=0}^nx^*_i(x)T^ie$ for some~$n$,
  so that
  \begin{multline*}
    \abs{f_k(x)}=
    \Bigabs{f_k\Bigl(\sum_{i=0}^nx^*_i(x)T^ie\Bigr)}\le
    \Bigl(\sum_{i=0}^n\norm{x^*_i}c^{(k)}_i\Bigr)\norm{x}\\=
    \Bigl(\sum_{i=k}^{n} r_i c_{i-k}\Bigr)\norm{x}\le
    \Bigl(\sum_{i=k}^\infty r_i c_{i-k}\Bigr)\norm{x}=
    \beta_k\norm{x},
  \end{multline*}
  so that $\norm{f_k}\le\beta_k$.  Hence, $f_k$ can be extended by
  continuity to a bounded functional on $[T^ie]_{i=1}^\infty$, and
  then by Hahn-Banach to a bounded functional on all of~$X$.

  Now we show that each $f_k$ annihilates~$Y$. Fix~$k$. Recall that for each 
  $\lambda\in \rho(T)^{-1}$ such that
  $\abs{\lambda}<\frac{1}{r(T)}$  we have 
  $h(\lambda,e)=\lambda\sum\limits_{i=0}^{\infty}\lambda^iT^ie$. 
  Therefore
  \begin{displaymath}
    f_k\bigl(h(\lambda,e)\bigr)=
    f_k\Bigl(\lambda\sum_{i=0}^\infty\lambda^iT^ie\Bigr)=
    \lambda\sum_{i=0}^\infty\lambda^ic^{(k)}_i=
    \lambda F_k(\lambda)=\lambda^{k+1}F(\lambda).
  \end{displaymath}
  for every $\lambda\in {\mathcal C}$ such that $\abs{\lambda} < \frac{1}{r(T)}$ 
  (recall $0 \in \overline{\mathcal C}$). The map
  $\lambda\mapsto h(\lambda,e)$ and, therefore, the map $\lambda\mapsto
  f_k\bigl(h(\lambda,e)\bigr)$, is analytic on the set 
  $\rho(T)^{-1}$. 
  Therefore, by the principle of uniqueness of analytic function,
  the functions $f_k\bigl(h(\lambda,e)\bigr)$ and $\lambda^{k+1}F(\lambda)$ must 
  agree on~${\mathcal C}$.
  Since $\lambda_n \in {\mathcal C}$ for all~$n$, we have
  $f_k\bigl(h(\lambda_n,e)\bigr)=\lambda_n^{k+1}F(\lambda_n)=0$ for
  all~$n$. Thus, $Y$ is annihilated by every~$f_k$.
  
  It is left to prove the linear independence of $\{f_k\}_{k=1}^\infty$. 
  Observe that $f_k\ne 0$ for all $k$ since $f_k(T^ie)\ne 0$ for $i\ge k$.
  Suppose that $f_N=\sum_{k=M}^{N-1}a_kf_k$ with $a_M\ne 0$. However
  $f_N(T^Me)=0$ by definition of $f_N$ while 
  $\sum_{k=M}^{N-1}a_kf_k(T^Me)=a_Mc_0\ne 0$, contradiction.
\end{proof}

\begin{remark}\label{q-nilp}
  Note that condition~\eqref{rho} of Theorem~\ref{main} is satisfied
  by many important classes of operators. For example, it is satisfied
  if $\sigma(T)$ is finite (in particular, if $T$ is quasinilpotent)
  or if $0$ belongs in the unbounded component of $\rho(T)$.
\end{remark}

\begin{corollary}\label{aihs-shifts}
  If $X= \ell_p$ ($1 \le p < \infty$) or $c_0$ and $T \in
  {\mathcal L}(X)$, is a weighted right shift operator with weights
  converging to zero then both $T$ and $T^*$ have almost invariant
  half-spaces.
\end{corollary}

\begin{proof}
%  Suppose that $T$ is a right shift operator on the given space that
%  we will denote by $X$. Let $(w_i)_{i=1}^\infty$ be the weights of
%  $T$. Let's prove that $T$ is quasinilpotent. Indeed, if
%   $x=(x_k)_{k=1}^\infty\in X$ then
%   \begin{displaymath}
%     T^nx=\Big(\,\underset{n}{\underbrace{0,\dots,0}}\,,\prod\limits_{k=1}^nw_kx_1,
%     \prod\limits_{k=1}^nw_{k+1}x_2,\prod\limits_{k=1}^nw_{k+2}x_3,\dots\Big).
%   \end{displaymath}
% Thus 
% $$
% \| T^n \| = \sup_{i \geq 0} \left| \prod_{k=1}^n w_{k+i} \right| .
% $$
% For $\varepsilon >0$ let $m_0 \in {\mathbb N}$ such that $|w_m|< \varepsilon$ for 
% $m \geq m_0$. Then for $n \geq m_0$ and $i \geq 0$ we have 
% $$
% \left| \prod_{K=1}^n w_{k+i} \right| \leq \prod_{k=i+1}^{m_0} |w_k| \varepsilon^{n-m_0}
% \leq \left( \max_{1 \leq \ell \leq m_0} \prod_{k=\ell}^{m_0} |w_k| \right) 
% \varepsilon^{n-m_0} =:C \varepsilon^n.
% $$
% Thus $\| T^n \|^{1/n} \leq C^{1/n} \varepsilon< 2 \varepsilon$ for large $n$.
% Hence, $T$ is quasinilpotent.
  It can be easily verified that $T$ is quasinilpotent.  Clearly, $T$
  has no eigenvalues, and the orbit of $e_1$ is evidently a minimal
  sequence. By Theorem~\ref{main} and Remark~\ref{q-nilp}, $T$ has
  almost invariant half-spaces. Finally,
  Proposition~\ref{aihs-adjoints} yields almost invariant half-spaces
  for~$T^*$.
\end{proof}

The following statement is a special case of
Corollary~\ref{aihs-shifts}.

\begin{corollary}
  If $D$ is a Donoghue operator then both $D$ and $D^*$ have almost
  invariant half-spaces.
\end{corollary}

Recall that a subset ${\mathcal D}$ of ${\mathbb C}$ is called a cone
if ${\mathcal D}$ is closed under addition and multiplication by
positive scalars.

\begin{remark}
  Condition~\eqref{rho} in Theorem~\ref{main} can be
  weakened as follows: instead of requiring that $\rho(T)$ contains a
  punctured disk centered at zero, we may only require that it
  contains a non-trivial sector of this disk, i.e., the intersection
  of the punctured disk with a non-empty open cone. Equivalently, 
  $\rho(T)^{-1}$ has a connected component $\mathcal C$ such that
  $0\in\overline{\mathcal C}$, and there exists an
  open cone ${\mathcal D}$ in ${\mathbb C}$ and $M>0$ such that
  $\bigl\{ z \in {\mathcal D}\mid \abs{z} \ge M \bigr\} \subseteq
  {\mathcal C}$.  Indeed, suppose that such $\mathcal D$ and $M$
  exist. Choose $\nu$ in this cone with $\abs{\nu}=M$.  Also, suppose
  that, as in the proof of Theorem~\ref{main}, we have already found a
  sequence $(\lambda_m)$ of zeros of~$F$. The set $\{M
  \frac{\lambda_m}{\abs{\lambda_m}}\}_{m=0}^\infty$ has an
  accumulation point, say~$\mu$. By passing to a subsequence, we may
  assume that $M \frac{\lambda_m}{\abs{\lambda_m}}\to\mu$.  Note that
  the spectrum of $\frac{\mu}{\nu}T$ is obtained by rotating the
  spectrum of~$T$. Thus $\rho\left(\frac{\mu}{\nu}T\right)^{-1}$ has a
  connected component ${\mathcal C}'$ such that $0 \in
  \overline{{\mathcal C}'}$ and there exists an open cone
  ${\mathcal D}'$ in ${\mathbb C}$ which contains a neighborhood of
  $\mu$ and $\bigl\{ z \in {\mathcal D}'\mid \abs{z} \ge M\bigr\} \subseteq
  {\mathcal C}'$. Thus by passing to a subsequence of $(\lambda_m)$ we
  can assume that $\lambda_m \in {\mathcal C}'$ for all~$m$.  Replace
  in the proof $T$ with $\frac{\mu}{\nu}T$. Note that this doesn't
  affect the assumptions on the operator and the definitions of
  $c_i$'s, $F$, and $\lambda_m$'s.  Finally, multiplying an operator
  by a non-zero number does not affect its almost invariant
  half-spaces.
\end{remark}

Note that every operator $T$ with $\sigma(T)\subseteq\mathbb R$
satisfies this weaker version of~condition~\eqref{rho}. In particular,
it is satisfied by self-adjoint operators on Hilbert spaces.

\begin{corollary}
  Suppose that $T\in\mathcal L(X)$ such that $T$ has no eigenvectors,
  $\sigma(T)\subseteq\mathbb R$, and there is a vector whose orbit is
  a minimal sequence. Then $T$ has an almost invariant half-space.
\end{corollary}

\section{Non-quasinilpotent operators} \label{ext}

In this section we will modify the argument of Theorem~\ref{main} to extend
its statement to another class of operators having non-zero spectral
radius. It is a standard fact that if $(x_i)$ is a minimal sequence
then
\begin{math}
  \frac{1}{\norm{x_n^*}}=\dist\bigl(x_n,[x_i]_{i\ne n}\bigr)
\end{math}
for every~$n$.

\begin{theorem}\label{extension}
  Let $X$ be a Banach space and $T\in {\mathcal L}(X)$ be an operator
  with $r(T)\le 1$
  having no eigenvectors. Let $e\in X$; put $x_n=T^ne$ for $n\in\mathbb
  N$. If $(x_n)$ is a minimal sequence and
  $\sum_{n=1}^\infty\frac{\norm{x_n^*}}{n}<\infty$ then $T$ has an
  almost invariant half-space.
\end{theorem}

\begin{proof}
  Let $D$ stands for the unit disk in~$\mathbb C$. For a sequence
  $(\lambda_n) \subset D$ such that
  \begin{equation} \label{lambda}
   \sum\limits_{n=1}^\infty\bigl(1-\abs{\lambda_n}\bigr)<\infty.
  \end{equation}
  the corresponding Blaschke product is defined by 
  \begin{equation} \label{Blaschke}
    B(z)= \prod_{n=1}^\infty \frac{\abs{\lambda_n}}{\lambda_n} 
    \frac{\lambda_n -z}{1-\overline{\lambda_n}\lambda_n}.
  \end{equation}
  It is well known that $B$ is a bounded analytic function on $D$ with
  zeros exactly at $(\lambda_n)$.  According to
  \cite[Theorem~2]{Newman:62} we can choose a sequence
  $(\lambda_n)\subset D$ satisfying~\eqref{lambda} such that
  $\frac{B^{(n)}(0)}{n!} = \mbox{O}\bigl(\frac{1}{n+1}\bigr)$.  Thus
  $B^{(n)}(0) =\mbox{O}\bigl(\frac{n!}{n+1}\bigr)$.  For $m \in
  {\mathbb N}$ set $F_m(z)= z^mB(z)$. Obviously the functions $(F_m)$
  are linearly independent. It is easy to see that for some $C>0$ we
  have
  \begin{equation} \label{coef}
    F_m^{(n)}(0)=
    \begin{cases}
      0 & \mbox{ for }n<j \\ 
      \frac{n!}{m!}B^{(n-m)}(0) \le
      C \frac{n!}{n-m+1}  & \mbox{ for }n \ge m.
    \end{cases}
  \end{equation}
  Put $Y=[h(\lambda_n,e)]_{n=1}^\infty$. By
  Lemma~\ref{h_lambda-ai}, $Y$ is almost invariant under $T$ and $\dim
  Y=\infty$ by Lemma~\ref{lin-indep}. As in the proof of
  Theorem~\ref{main}, we will show that under the conditions of the
  Theorem~\ref{extension} there is a sequence of linearly independent
  functionals annihilating~$Y$.

  Define a linear functional $f_m$ on $\Span x_n$ by
  $f_m(x_n)=\frac{F_m^{(n)}(0)}{n!}$. Since $T$ has no eigenvectors,
  the orbit of $T$ is linearly independent, so $f_m$ is well defined.

  Let's prove that $f_m$ is bounded for every $m\in\mathbb N$. Take
  any $x:=\sum\alpha_nx_n\in\Span x_n$. Using~\eqref{coef}, we obtain
  \begin{multline*}
    \bigabs{f_m(x)}=
    \Bigabs{\sum\alpha_n\frac{F_m^{(n)}(0)}{n!}}\le
    C\sum_{n\ge m}\frac{\abs{\alpha_n}}{n-m+1}\\
    =C\sum_{n\ge m}\frac{\abs{x_n^*(x)}}{n-m+1}
    \le C\norm{x}\sum_{n\ge m}\frac{\norm{x_n^*}}{n-m+1}.
  \end{multline*}
  It suffices to show that $\sum_{n\ge m}\frac{\norm{x_n^*}}{n-m+1}<\infty$.
  Note that
  \begin{displaymath}
    m(n-m+1)=(m-1)(n-m)+n\ge n
  \end{displaymath}
  whenever $n\ge m$, so that 
  \begin{displaymath}
    \sum_{n=m}^{\infty}\frac{\norm{x_n^*}}{n-m+1}
    =m\sum_{n=m}^{\infty}\frac{\norm{x_n^*}}{m(n-m+1)}
    \le m\sum_{n=m}^{\infty}\frac{\norm{x_n^*}}{n}<\infty
  \end{displaymath}
  by assumption. Hence, $f_m$ is bounded, so that we can extend it to~$X$.
% \begin{eqnarray*}
% \left| f_j\left(\sum_{n \geq 0} \alpha_n T^ne \right) \right| & = & 
% \left| \sum_{n \geq 0} \alpha_n  \frac{F_j^{(n)}(0)}{n!} \right| 
% \leq C \sum_{n \geq j} |\alpha_n | \frac{1}{n-j+1} \quad \mbox{(by (\ref{coef}))}\\ 
% & = & C \sum_{n \geq j} \frac{|\alpha_n|}{ \mu_n }
% \frac{\mu_n}{n-j+1} 
% \leq  C \left( \max_{n \geq 0}  \frac{|\alpha_n|}{\mu_n}   \right)
% \sum_{n \geq j } \frac{\mu_n}{n-j+1 } \\
% & \leq & C'  \left\| \sum_{n \geq 0} \frac{\alpha_n}{\mu_n}  
% \mu_n T^n e \right\| 
% \sum_{n \geq j } \frac{\mu_n}{n-j+1} 
% \quad \mbox{(by (\ref{c0}))}\\
% & = & C'  \left\| \sum_{n \geq 0} \alpha_n  T^n e \right\|
% \sum_{n \geq j } \frac{\mu_n}{n-j+1 }.
% \end{eqnarray*}
  Observe that if $\abs{\lambda}<1$ and $m\in\mathbb N$
  then~\eqref{Neumann} yields that
  \begin{displaymath}
    f_m\big(h(\lambda,e)\big)
    =f_m\Big(\lambda\sum\limits_{n=0}^\infty\lambda^nT^ne\Big)
    = \lambda \sum_{n=0}^\infty \lambda^n \frac{F_m^{(n)}(0)}{n!}
    =\lambda F_m(\lambda),
  \end{displaymath}
  Thus $f_m\big(h(\lambda_k,e)\big)=\lambda_kF_m(\lambda_k)=0$ for all
  $m,k\in\mathbb N$, hence each $f_m$ annihilates~$Y$.

  Finally, the set $\{f_m\}_{m=1}^\infty$ is linearly independent.
  Indeed if it was linearly dependent and a certain linear non-zero
  linear combination of them vanishes, then by writing the Taylor
  expansion of each $F_m$ on $D$ we see that the same linear
  combination of $F_m$'s would vanish. This is a contradiction, since
  the $F_m$'s are linear independent.
\end{proof}

\section{Invariant subspaces of operators with many almost invariant
  half-spaces}

Let $X$ be a Banach space and $T\colon X\to X$ be a bounded operator.
It is well known that if every subspace of $X$ is invariant under $T$
then $T$ must be a multiple of identity. In this section we will
obtain a result of the same spirit for almost invariant half-spaces.

\begin{proposition}\label{rich-aihs}
  Let $X$ be a Banach space and $T \in {\mathcal L}(X)$.  Suppose that
  every half-space of $X$ is almost invariant under~$T$. Then $T$ has
  a non-trivial invariant subspace of finite codimension. Iterating,
  one can get a chain of such subspaces.
\end{proposition}

\begin{proof}
  Let's assume that $T$ has no non-trivial invariant subspaces of
  finite codimension. We will now construct by an inductive procedure
  a half-space that is not almost invariant under~$T$.

  Put $Y_0=X$. Fix an arbitrary non-zero $z_1\in X$.  Choose $f_1\in
  X^*$ such that $f_1(z_1)\ne 0$ and put $Y_1=\ker f_1$.

  Since $Y_1$ is not invariant under~$T$, there exists $z_2\in Y_1$
  such that $f_1(Tz_2)\ne 0$. Define $g_2\in Y_1^*$ by
  $g_2(y)=f_1(Ty)$.  Let $P_1$ be a projection along
  $\Span\left\{z_1\right\}$ onto~$Y_1$. Define $f_2=g_2\circ P_1\in
  X^*$.  Now put $Y_2=\ker g_2=\ker f_2\cap Y_1$. Then we have
  $Y_1=Y_2\oplus\Span\left\{z_2\right\}$.  Since $f_1(Ty)=g_2(y)=0$
  for all $y\in Y_2$, we have $TY_2\subseteq Y_1$.

  Continuing inductively with this procedure, we will build sequences
  $(z_n)$ of vectors, $(f_n)$ of functionals, and $(Y_n)$ of subspaces
  such that
  \begin{enumerate}
    \item\label{first} $z_{n+1}\in Y_{n}$,
    \item $Y_{n+1}=\ker f_{n+1}\cap Y_n=\bigcap\limits_{k=1}^{n+1}\ker f_k$,
    \item\label{three} $f_{n+1}(y)=f_n(Ty)$ for all $y\in Y_n$,
    \item\label{four} $Y_n=Y_{n+1}\oplus{\rm span}\left\{z_{n+1}\right\}$,
    \item $TY_{n+1}\subseteq Y_n$, and
    \item\label{last} $f_n(z_i)=0\Leftrightarrow i\ne n$,
  \end{enumerate}
  for all $n\in\mathbb N$. Indeed, suppose we have defined~$Y_i$,
  $z_i$, and $f_i$, $1\le i\le n$,
  satisfying~\eqref{first}--\eqref{last}. Define $g_{n+1}\in Y_n^*$ by
  $g_{n+1}(y)=f_n(Ty)$ and put $f_{n+1}=g_{n+1}\circ P_n\in X^*$ where
  $P_n$ is a projection along $[z_k]_{k=1}^n$ onto $Y_n$ \big(taking
  $P_n(x)=x-\sum_{k=1}^{n}\frac{f_k(x)}{f_k(z_k)}z_k$ will do
  the job\big).  Again, there is $z_{n+1}\in Y_n$ such that
  $f_{n+1}(z_{n+1})\ne 0$.  Put $Y_{n+1}=\ker f_{n+1}\cap Y_n$.
  Evidently, \eqref{first}--\eqref{last} are then satisfied.

  It is easily seen that the sequence $(z_k)$ is linearly independent.
  Put $Z=[z_{2k}]_{k=1}^\infty$. Clearly $\dim Z=\infty$. It is
  also easy to see that $f_{2k-1}|_Z=0$ for all $k\in\mathbb N$.
  Thus, $Z$ is actually a half-space.

  By assuption of the theorem, there exists $F$ with $\dim F=m<\infty$
  such that $TZ\subseteq Z+F$. For each $k\in\mathbb N$, pick $u_k\in
  Z$ and $v_k\in F$ such that $Tz_{2k}=u_k+v_k$. By~\eqref{four}, we
  have $z_{2k}\in Y_{2k-1}$. Applying~\eqref{three} with $n=2k-1$, we
  get $f_{2k}(z_{2k})=f_{2k-1}(Tz_{2k})$. Now~\eqref{last} yields
  $f_{2k-1}(Tz_{2k})\ne 0$.

  On the other hand, if $1\le i<k$ then $z_{2k}\in Y_{2i-1}$, so that
  analogously $f_{2i}(z_{2k})=f_{2i-1}(Tz_{2k})$. Therefore
  $f_{2i-1}(Tz_{2k})=0$.

  Since $f_{2j-1}|_Z=0$ for all $j\in\mathbb N$, we have
  $f_{2j-1}(u_k)=0$ for all $j$ and~$k$. Therefore, for all
  $k\in\mathbb N$ and $1\le i<k$, we have $f_{2k-1}(v_k)\ne 0$ and
  $f_{2i-1}(v_k)=0$. This implies, however, that $F$ is infinite
  dimensional.
\end{proof}

Using a similar technique, we obtain the following result.

\begin{proposition}
  For all $T\in\mathcal L(X)$ and $n\in\mathbb N$ then there exists a
  subspace $Y$ of $X$ with $\codim Y=n$ and a vector $e_Y\in X$ such
  that $TY\subseteq Y+\Span\{e_Y\}$.
\end{proposition}

% Indeed, if $n=1$ then the statement is trivial. Suppose that the statement is valid 
% for every $k<n$. Start the same procedure as in the proof of Proposition~\ref{rich-aihs}
% of constructing the sequence $(Y_n)$. Observe that $Y_n$ from that procedure has 
% codimension $n$ and $TY_n\subseteq Y_n+\Span\{e\}$ for some $e\in X$. If $Y_k$ 
% is not invariant under $T$ for every $k<n$ then $Y_n$ is the subspace we are looking for.
% Suppose that $Y_k$ is invariant under $T$ for some $k<n$ (so that the procedure cannot
% be continued after $k$-th step). By induction assumption, there is a subspace $Y$ of $Y_k$ having 
% codimension $(n-k)$ in $Y_k$ such that $TY\subseteq Y+\Span\{e\}$ for some $e\in X$.
% To finish the proof, observe that $Y$ has codimension $n$ in $X$.

\begin{proof}
  Proof is by induction on~$n$. For $n=1$, any hyperplane satisfies the
  conclusion of the statement. Suppose that the statement is valid for
  all $k<n$.

  Suppose that $X$ contains a subspace $Y$ of codimension $j\le n$
  that is invariant under~$T$. If $j=n$ then we are done. If $j<n$
  then by the induction assumption we can find $Z\subseteq Y$ such
  that $Z$ has codimension $n-j$ in $Y$ and $TZ\subseteq Z+[y]$ for
  some $y\in Y$. Indeed, consider the restriction $T'$ of $T$ to~$Y$.
  Now we apply the induction assumption to $T'$ and to $n-j$ and
  produce a subspace $Z\subseteq Y$ invariant under $T$ of
  codimension~$j$.  But then $Z$ has codimension $n$ in $X$ and still
  $TZ\subseteq Z+[y]$, so that $Z$ satisfies the conclusion.

  Therefore, we can assume that $Z$ has no invariant subspaces of
  codimension $k\le n$. Thus we can use the argment of Proposition~5.1
  to show that there exist (finite) sequences of vectors
  $(z_k)_{k=1}^{n+1}$, functionals $(f_k)_{k=1}^{n}$, and subspaces
  $(Y_k)_{k=1}^n$ such that the conditions~\eqref{first}--\eqref{last}
  are satisfied. In particular, we get:
  \begin{displaymath}
   Y_n=\bigcap\limits_{k=1}^n\ker f_k,
  \end{displaymath}
  and $(f_k)$ are linearly independent, so that ${\rm codim}\,Y_n=n$.
  Finally, by~\eqref{last} and~\eqref{four}, we have  $TY_n\subseteq
  Y_{n-1}=Y_n+[z_n]$.
\end{proof}

\medskip

\textbf{Acknowledgments.} We would like to thank Heydar Radjavi for
helpful discussions and suggestions. A part of the work on this paper
was done while the first and the fourth authors were attending SUMIRFAS
in the summer of 2008; we would like to express our thanks to
its organizers.

\end{document}